\documentclass[11pt]{amsart}

\usepackage{amsmath}
\usepackage{amssymb}
\usepackage{amsfonts}
\usepackage{mathrsfs}

\vfuzz2pt \DeclareMathOperator{\RE}{Re}
\DeclareMathOperator{\IM}{Im}

\DeclareMathOperator{\Psl}{PSL} 

 \DeclareMathOperator{\Ad}{Ad}
 \DeclareMathOperator{\sign}{sign}
\DeclareMathOperator{\TRR}{Tr} \DeclareMathOperator{\MOD}{mod}
\DeclareMathOperator{\vol}{vol} 

\DeclareMathOperator{\II}{i}

\newenvironment{teorem}[2][Theorem]{\begin{trivlist}
\item[\hskip \labelsep {\bfseries #1}\hskip \labelsep {\bfseries #2}]}{\end{trivlist}}

\newenvironment{lem}[2][Lemma]{\begin{trivlist}
\item[\hskip \labelsep {\bfseries #1}\hskip \labelsep {\bfseries #2}]}{\end{trivlist}}
\newenvironment{cor}[2][Corollary]{\begin{trivlist}
\item[\hskip \labelsep {\bfseries #1}\hskip \labelsep {\bfseries #2}]}{\end{trivlist}}

\newenvironment{remm}[2][Remark]{\begin{trivlist}
\item[\hskip \labelsep {\bfseries #1}\hskip \labelsep {\bfseries #2}]}{\end{trivlist}}

\numberwithin{equation}{section}
\numberwithin{theorem}{section}

\newenvironment{rem}[2][Acknowledgement.]{\begin{trivlist}
\item[\hskip \labelsep {\bfseries #1}\hskip \labelsep {\bfseries #2}]}{\end{trivlist}}

\newcommand{\set}[1]{\left\{#1\right\}}
\newcommand{\abs}[1]{\left\vert#1\right\vert}
\newcommand{\br}[1]{\left(#1\right)}
\newcommand{\SqBr}[1]{\left[#1\right]}

\begin{document}

\title[Distribution of singularities]{Distribution of the zeta functions singularities for compact even-dimensional locally symmetric spaces}
\author{Muharem Avdispahi\'c and D\v zenan Gu\v si\'c}

\address{University of Sarajevo, Department of Mathematics, Zmaja od Bosne
35, 71000 Sarajevo, Bosnia and Herzegovina}
\email{\textbf{mavdispa@pmf.unsa.ba}}

\address{University of Sarajevo, Department of Mathematics, Zmaja od Bosne
35, 71000 Sarajevo, Bosnia and Herzegovina}
\email{\textbf{dzenang@pmf.unsa.ba}}



\maketitle

\begin{abstract}
For compact, even-dimensional, locally symmetric spaces, we obtain precise estimates on the number of singularities of Selberg's and Ruelle's zeta functions considered by U. Bunke and M. Olbrich.
\end{abstract}

\section{Introduction}

Let $Y=\Gamma\backslash G/K=\Gamma\backslash X$ be a compact, $n-$ dimensional ($n$ even), locally
symmetric Riemannian manifold with negative sectional curvature, where $G$ is
a connected semisimple Lie group of real rank one, $K$ is a
maximal compact subgroup of $G$ and $\Gamma$ is a discrete
co-compact torsion free subgroup of $G$.

We require $G$ to be linear in order to have complexification
available.

In \cite{Bunke}, Bunke and Olbrich studied the zeta functions of
Selberg and Ruelle associated with a locally homogeneous vector
bundle of the unit sphere bundle of $Y$.

It is well known that the number of the zeros $\rho$, $\abs{\rho}\leq t$ of an entire function of order not larger than $m$ is $O\br{t^m}$
(see, \cite[p.~510]{Fried}). The classical Selberg zeta function is an entire function of order two (see, \cite{Selberg}). However, the number of it's zeros on the interval $\frac{1}{2}+\II x$, $0<x<t$ is $\frac{A}{4\pi}t^2+O\br{t}$, for some explicitly determined constant $A$ (see, \cite{Hejhal1}-\cite{Randol2}). The main purpose of this paper is to give
an analogous result for the zeta functions considered in
\cite{Bunke}.
\newline
\section{Preliminaries}

In the sequel we follow the notation of \cite{Bunke} (see also \cite{AG}).

Let $\mathfrak{g}=\mathfrak{k}\oplus\mathfrak{p}$ be the Cartan
decomposition of the Lie algebra $\mathfrak{g}$ of $G$,
$\mathfrak{a}$ a maximal abelian subspace of $\mathfrak{p}$ and
$M$ the centralizer of $\mathfrak{a}$ in $K$ with Lie algebra
$\mathfrak{m}$. We normalize the $\Ad{\br{G}}-$ invariant inner
product $(.,.)$ on $\mathfrak{g}$ to restrict to the metric on
$\mathfrak{p}$. Let $SX=G/M$ be the unit sphere bundle of $X$.
Hence $SY=\Gamma\backslash G/M$.

Let $\Phi\br{\mathfrak{g},\mathfrak{a}}$ be the root system and
$W=W\br{\mathfrak{g},\mathfrak{a}}\cong\mathbb{Z}_{2}$ its Weyl
group. Fix a system of positive roots
$\Phi^{+}\br{\mathfrak{g},\mathfrak{a}}\subset\Phi\br{\mathfrak{g},\mathfrak{a}}$.
Let

\[\mathfrak{n}=\sum\limits_{\alpha\in\Phi^{+}\br{\mathfrak{g},\mathfrak{a}}}\mathfrak{n}_{\alpha}\]
\newline
be the sum of the root spaces corresponding to elements of
$\Phi^{+}\br{\mathfrak{g},\mathfrak{a}}$. The decomposition
$\mathfrak{g}=\mathfrak{k}\oplus\mathfrak{p}\oplus\mathfrak{n}$
corresponds to the Iwasawa decomposition of the group $G=KAN$.
Define $\rho\in\mathfrak{a}_{\mathbb{C}}^{*}$ by

\[\rho=\frac{1}{2}\sum\limits_{\alpha\in\Phi^{+}\br{\mathfrak{g},\mathfrak{a}}}\dim\br{\mathfrak{n}_{\alpha}}\alpha.\]
\newline
The positive Weyl chamber $\mathfrak{a}^{+}$ is the half line in
$\mathfrak{a}$ on which the positive roots take positive values.
Let $A^{+}=\exp\br{\mathfrak{a}^{+}}\subset A$.

The symmetric space $X$ has a compact dual space $X_{d}=G_{d}/K$,
where $G_{d}$ is the analytic subgroup of $GL\br{n,\mathbb{C}}$
corresponding to
$\mathfrak{g}_{d}=\mathfrak{k}\oplus\mathfrak{p}_{d}$,
$\mathfrak{p}_{d}=$i$\mathfrak{p}$. We normalize the metric on
$X_{d}$ in such a way that the multiplication by i induces an
isometry between $\mathfrak{p}$ and $\mathfrak{p}_{d}$.

Let $i^{*}:R\br{K}\rightarrow R\br{M}$ be the restriction map
induced by the embedding $i:M\hookrightarrow K$, where $R\br{K}$
and $R\br{M}$ are the representation rings over $\mathbb{Z}$ of
$K$ and $M$, respectively.

Since $n$ is even, every $\sigma\in\hat{M}$ is invariant under the
action of the Weyl group $W$ (see, \cite[p.~27]{Bunke}). Let
$\sigma\in\hat{M}$. We choose $\gamma\in R\br{K}$ such that
$i^{*}\br{\gamma}=\sigma$ and represent it by $\Sigma
a_{i}\gamma_{i}$, $a_{i}\in\mathbb{Z},\gamma_{i}\in\hat{K}$. Set

\[V_{\gamma}^{\pm}=\sum\limits_{\sign{\br{a_{i}}}=\pm1}\sum\limits_{m=1}^{\abs{a_{i}}}V_{\gamma_{i}},\]
\newline
where $V_{\gamma_{i}}$ is the representation space of
$\gamma_{i}$. Define
$V\br{\gamma}^{\pm}=G\times_{K}V_{\gamma}^{\pm}$ and
$V_{d}\br{\gamma}^{\pm}=G_{d}\times_{K}V_{\gamma}^{\pm}$. To
$\gamma$ we associate $\mathbb{Z}_{2}-$ graded homogeneous vector
bundles $V\br{\gamma}=V\br{\gamma}^{+}\oplus V\br{\gamma}^{-}$ and
$V_{d}\br{\gamma}=V_{d}\br{\gamma}^{+}\oplus V_{d}\br{\gamma}^{-}$
on $X$ and $X_{d}$, respectively. Let

\[V_{Y,\chi}\br{\gamma}=\Gamma\backslash\br{V_{\chi}\otimes
V\br{\gamma}}\]
\newline
be a $\mathbb{Z}_{2}-$ graded vector bundle on $Y$, where
$\br{\chi,V_{\chi}}$ is a finite-dimensional unitary
representation of $\Gamma$.

By \cite[p.~36]{Bunke},
\begin{equation}\label{1111}
\frac{\chi\br{Y}}{\chi\br{X_{d}}}=\br{-1}^{\frac{n}{2}}\frac{\vol{\br{Y}}}{\vol{\br{X_{d}}}}.
\end{equation}\

Reasoning as in the beginning of Subsection 1.1.2 in \cite{Bunke},
we choose a Cartan subalgebra $\mathfrak{t}$ of $\mathfrak{m}$ and
a system of positive roots
$\Phi^{+}\br{\mathfrak{m}_{\mathbb{C}},\mathfrak{t}}$. Then,
$\rho_{\mathfrak{m}}\in$i$\mathfrak{t}^{*}$, where
\[\rho_{\mathfrak{m}}=\frac{1}{2}\sum\limits_{\alpha\in\Phi^{+}\br{\mathfrak{m}_{\mathbb{C}},\mathfrak{t}}}\alpha.\]
\newline
Let $\mu_{\sigma}\in$i$\mathfrak{t}^{*}$ be the highest weight of
$\sigma$. Set

\[c\br{\sigma}=\abs{\rho}^{2}+\abs{\rho_{\mathfrak{m}}}^{2}-\abs{\mu_{\sigma}+\rho_{\mathfrak{m}}}^{2},\]
\newline
where the norms are induced by the complex bilinear extension to
$\mathfrak{g}_{\mathbb{C}}$ of the inner product $(.,.)$. Finally,
we introduce the operators  (see, \cite[p.~28]{Bunke})

\[A_{d}\br{\gamma,\sigma}^{2}=\Omega +c\br{\sigma}: C^{\infty}\br{X_{d},V_{d}\br{\gamma}}\rightarrow C^{\infty}\br{X_{d},V_{d}\br{\gamma}},\]
\[A_{Y,\chi}\br{\gamma,\sigma}^{2}=-\Omega -c\br{\sigma}: C^{\infty}\br{Y,V_{Y,\chi}\br{\gamma}}\rightarrow C^{\infty}\br{Y,V_{Y,\chi}\br{\gamma}},\]
\newline
$\Omega$ being the Casimir element of the complex
universal enveloping algebra $\mathcal{U}\br{\mathfrak{g}}$ of
$\mathfrak{g}$.

Let $m_{\chi}\br{s,\gamma,\sigma}=\TRR{E_{A_{Y,\chi}\br{\gamma,\sigma}}}\br{\set{s}}$,
$m_{d}\br{s,\gamma,\sigma}=\TRR{E_{A_{d}\br{\gamma,\sigma}}}\br{\set{s}}$,
where $E_{A}\br{.}$ denotes the family of spectral projections of
a normal operator $A$.

Now, we choose a maximal abelian subalgebra $\mathfrak{t}$ of
$\mathfrak{m}$. Then,
$\mathfrak{h}=\mathfrak{t}_{\mathbb{C}}\oplus\mathfrak{a}_{\mathbb{C}}$
is a Cartan subalgebra of $\mathfrak{g}_{\mathbb{C}}$. Let
$\Phi^{+}\br{\mathfrak{g}_{\mathbb{C}},\mathfrak{h}}$ be a
positive root system having the property that, for
$\alpha\in\Phi\br{\mathfrak{g}_{\mathbb{C}},\mathfrak{h}}$,
$\alpha_{|\mathfrak{a}}\in\Phi^{+}\br{\mathfrak{g},\mathfrak{a}}$
implies
$\alpha\in\Phi^{+}\br{\mathfrak{g}_{\mathbb{C}},\mathfrak{h}}$.
Let
\[\delta=\frac{1}{2}\sum\limits_{\alpha\in\Phi^{+}\br{\mathfrak{g}_{\mathbb{C}},\mathfrak{h}}}\alpha.\]
\newline
We set $\rho_{\mathfrak{m}}=\delta-\rho$. Define the root vector
$H_{\alpha}\in\mathfrak{a}$ for
$\alpha\in\Phi^{+}\br{\mathfrak{g},\mathfrak{a}}$ by

\[\lambda\br{H_{\alpha}}=\frac{\br{\lambda,\alpha}}{\br{\alpha,\alpha}},\]
where $\lambda\in\mathfrak{a}^{*}$.

For $\alpha\in\Phi^{+}\br{\mathfrak{g},\mathfrak{a}}$, we define
$\varepsilon_{\alpha}\br{\sigma}\in\set{0,\frac{1}{2}}$ by

\[e^{2\pi\II\varepsilon_{\alpha}\br{\sigma}}=\sigma\br{e^{2\pi
\II H_{\alpha}}}\in\set{\pm 1}.\]
\newline
According to \cite[p.~47]{Bunke}, the root system
$\Phi^{+}\br{\mathfrak{g},\mathfrak{a}}$ is of the form
$\Phi^{+}\br{\mathfrak{g},\mathfrak{a}}=\set{\alpha}$ or
$\Phi^{+}\br{\mathfrak{g},\mathfrak{a}}=\set{\frac{\alpha}{2},\alpha}$
for the long root $\alpha$. Let $\alpha$ be the long root in
$\Phi^{+}\br{\mathfrak{g},\mathfrak{a}}$. We set $T=\abs{\alpha}$.
For $\sigma\in\hat{M}$, $\epsilon_{\sigma}\in\set{0,\frac{1}{2}}$
is given by

\[\epsilon_{\sigma}\equiv\frac{\abs{\rho}}{T}+\varepsilon_{\alpha}\br{\sigma}\,\MOD{\mathbb{Z}}.\]
\newline
We define the lattice
$L\br{\sigma}\subset\mathbb{R}\cong\mathfrak{a}^{*}$ by
$L\br{\sigma}=T\br{\epsilon_{\sigma}+\mathbb{Z}}$. Finally, for
$\lambda\in\mathfrak{a}_{\mathbb{C}}^{*}\cong\mathbb{C}$ we set
\[P_{\sigma}\br{\lambda}=\prod\limits_{\beta\in\Phi^{+}\br{\mathfrak{g}_{\mathbb{C}},\mathfrak{h}}}
\frac{\br{\lambda
+\mu_{\sigma}+\rho_{\mathfrak{m}},\beta}}{\br{\delta,\beta}}.\]
\newline

Since $n$ is even, there exists a $\sigma-$ admissible $\gamma\in
R\br{K}$ for every $\sigma\in\hat{M}$ (see, \cite[p.~49, Lemma
1.18]{Bunke}). Here, $\gamma\in R\br{K}$ is called $\sigma-$
admissible if $i^{*}\br{\gamma}=\sigma$ and
$m_{d}\br{s,\gamma,\sigma}=P_{\sigma}\br{s}$ for all $0\leq s\in
L\br{\sigma}$.

\section{Zeta functions and the geodesic flow}

Since $\Gamma\subset G$ is co-compact and torsion free, there are
only two types of conjugacy classes - the class of the identity
$1\in\Gamma$ and classes of hyperbolic elements.

Let $g\in G$ be hyperbolic. Then there is an Iwasawa decomposition
$G=NAK$ such that $g=am\in A^{+}M$. Following \cite[p.~59]{Bunke},
we define

\[l\br{g}=\abs{\log\br{a}}.\]
\newline
Let $\Gamma_{\textrm{h}}$, resp. $\text{P}\Gamma_{\textrm{h}}$
denote the set of the $\Gamma -$ conjugacy classes of hyperbolic
resp. primitive hyperbolic elements in $\Gamma$.

Let $\varphi$ be the geodesic flow on $SY$ determined by the
metric of $Y$. In the representation $SY=\Gamma\backslash G/M$, $\varphi$
is given by

\[\varphi : \mathbb{R}\times SY\ni\br{t,\Gamma gM}\rightarrow\Gamma g\exp\br{-tH}M\in SY,\]
\newline
where $H$ is the unit vector in $\mathfrak{a}^{+}$. If
$V_{\chi}\br{\sigma}=\Gamma\backslash\br{G\times_{M}V_{\sigma}\otimes
V_{\chi}}$ is the vector bundle corresponding to
finite-dimensional unitary representations
$\br{\sigma,V_{\sigma}}$ of $M$ and $\br{\chi,V_{\chi}}$ of
$\Gamma$, then we define a lift $\varphi_{\chi,\sigma}$ of
$\varphi$ to $V_{\chi}\br{\sigma}$ by (see, \cite[p.~95]{Bunke})

\[\varphi_{\chi,\sigma} : \mathbb{R}\times V_{\chi}\br{\sigma}\ni\br{t,\SqBr{g,v\otimes w}}\rightarrow
\SqBr{g\exp\br{-tH},v\otimes w}\in V_{\chi}\br{\sigma}.\]
\newline
For $\RE{\br{s}}>2\rho$, the Ruelle zeta function for the flow
$\varphi_{\chi,\sigma}$ is defined by the infinite product

\[Z_{R,\chi}\br{s,\sigma}=\prod\limits_{\gamma_{0}\in\text{P}\Gamma_{\textrm{h}}}
\det\br{1-\br{\sigma\br{m}\otimes\chi\br{\gamma_{0}}}e^{-sl\br{\gamma_{0}}}}^{\br{-1}^{n-1}}.\]
\newline
The Selberg zeta function for the flow $\varphi_{\chi,\sigma}$ is
given by

\begin{equation}\label{3.0}
Z_{S,\chi}\br{s,\sigma}=
\end{equation}
\[\prod\limits_{\gamma_{0}\in\text{P}\Gamma_{\textrm{h}}}\prod\limits_{k=0}^{+\infty}
\det\br{1-\br{\sigma\br{m}\otimes\chi\br{\gamma_{0}}\otimes
S^{k}\br{\Ad{\br{ma}_{\bar{\mathfrak{n}}}}}}e^{-\br{s+\rho}l\br{\gamma_{0}}}},\]
\newline
for $\RE{\br{s}}>\rho$, where $S^{k}$ denotes the $k-$th symmetric
power of an endomorphism, $\bar{\mathfrak{n}}=\theta\mathfrak{n}$
is the sum of negative root spaces of $\mathfrak{a}$ as usual, and
$\theta$ is the Cartan involution of $\mathfrak{g}$.

Let $\mathfrak{n}_{\mathbb{C}}$ be the complexification of
$\mathfrak{n}$. For
$\lambda\in\mathbb{C}\cong\mathfrak{a}_{\mathbb{C}}^{*}$ let
$\mathbb{C}_{\lambda}$ denote the one-dimensional representation
of $A$ given by $A\ni a\rightarrow a^{\lambda}$. Let $p\geq 0$.
There exist sets
\[I_{p}=\set{\br{\tau,\lambda}\mid\tau\in\hat{M},\lambda\in\mathbb{R}}\]
\newline
such that $\Lambda^{p}\mathfrak{n}_{\mathbb{C}}$ as a
representation of $MA$ decomposes with respect to $MA$ as
\[\Lambda^{p}\mathfrak{n}_{\mathbb{C}}=\sum\limits_{\br{\tau,\lambda}\in I_{p}}V_{\tau}\otimes\mathbb{C}_{\lambda},\]
\newline
where $V_{\tau}$ is the space of the representation $\tau$. Bunke
and Olbrich proved that the Ruelle zeta function $Z_{R,\chi}\br{s,\sigma}$ has the following
representation (see, \cite[p.~99, Prop. 3.4]{Bunke})

\begin{equation}\label{3.1}
Z_{R,\chi}\br{s,\sigma}=\prod\limits_{p=0}^{n-1}\prod\limits_{\br{\tau,\lambda}\in I_{p}}Z_{S,\chi}\br{s+\rho-\lambda,\tau\otimes\sigma}^{\br{-1}^{p}}.
\end{equation}

Let $d_{Y}=-\br{-1}^{\frac{n}{2}}$. The following theorem holds
true (see, \cite[p.~113, Th. 3.15]{Bunke})
\newline
\begin{teorem}{A.}\label{t3.15}
\textit{The Selberg zeta function $Z_{S,\chi}\br{s,\sigma}$ has a
meromorphic continuation to all of $\mathbb{C}$. If $\gamma$ is
$\sigma-$admissible, then the singularities (zeros and poles) of
$Z_{S,\chi}\br{s,\sigma}$ are the following ones:\\
\begin{enumerate}
    \item at $\pm is$ of order $m_{\chi}\br{s,\gamma,\sigma}$ if
    $s\neq 0$ is an eigenvalue of
    $A_{Y,\chi}\br{\gamma,\sigma}$,\\
    \item at $s=0$ of order $2m_{\chi}\br{0,\gamma,\sigma}$ if
    $0$ is an eigenvalue of $A_{Y,\chi}\br{\gamma,\sigma}$,\\
    \item at $-s$, $s\in T\br{\mathbb{N}-\epsilon_{\sigma}}$ of
    order
    $2\frac{d_{Y}\dim\br{\chi}\vol\br{Y}}{\vol\br{X_{d}}}m_{d}\br{s,\gamma,\sigma}$. Then $s>0$ is an eigenvalue of $A_{d}\br{\gamma,\sigma}$.\\
\end{enumerate}
If two such points coincide, then the orders add up.}
\end{teorem}\

Since $\lambda$ runs through $\frak{a}-$weights on the exterior algebra of $\frak{n}$, the shifts $\rho-\lambda$ that appear in (\ref{3.1}) are always contained in the interval $\SqBr{-\rho,\rho}$. Hence, by (\ref{3.1}) and Theorem A, $s_R\in\mathbb{R}$ for all singularities $s_R$ of the $Z_{R,\chi}\br{s,\sigma}$ with $\abs{\RE{\br{s_R}}}>\rho$. 

In \cite{AG}, we proved that there exist entire functions $Z_{S}^{1}\br{s}$, $Z_{S}^{2}\br{s}$ of order at most $n$ such that
\begin{equation}\label{3.3}
Z_{S,\chi}\br{s,\sigma}=\frac{Z_{S}^{1}\br{s}}{Z_{S}^{2}\br{s}}.
\end{equation}\
\newline
Here, $\gamma$ is $\sigma\,-$ admissible, the zeros of $Z_{S}^{1}\br{s}$ correspond to the
zeros of $Z_{S,\chi}\br{s,\sigma}$ and the zeros of
$Z_{S}^{2}\br{s}$ correspond to the poles of
$Z_{S,\chi}\br{s,\sigma}$. The orders of the zeros of
$Z_{S}^{1}\br{s}$ resp. $Z_{S}^{2}\br{s}$ equal the orders of the
corresponding zeros resp. poles of $Z_{S,\chi}\br{s,\sigma}$.

\section{Results}

\begin{lem}{4.1.}\label{IV.1}
\textit{If $\gamma$ is $\sigma\,-$ admissible, then}\\
\[Z_{S,\chi}\br{s,\sigma}=e^{-K\int\limits_{0}^{s}P_{\sigma}\br{w}\left\{%
\begin{array}{ll}
    \tan\br{\frac{\pi w}{T}}, & \hbox{$\epsilon_{\sigma}=\frac{1}{2}$} \\
    -\cot\br{\frac{\pi w}{T}}, & \hbox{\hbox{$\epsilon_{\sigma}=0$}} \\
\end{array}%
\right\}dw}Z_{S,\chi}\br{-s,\sigma},\]
\newline
\textit{where $K=\frac{2\pi\dim\br{\chi}\chi\br{Y}}{\chi\br{X_{d}}T}$}.
\end{lem}

\begin{proof}
By \cite[p.~118, Th. 3.19]{Bunke},
$Z_{S,\chi}\br{s,\sigma}$ has the representation

\[Z_{S,\chi}\br{s,\sigma}=\det\br{A_{Y,\chi}\br{\gamma,\sigma}^{2}+s^{2}}
\det\br{A_{d}\br{\gamma,\sigma}+s}^{-\frac{2\dim\br{\chi}\chi\br{Y}}{\chi\br{X_{d}}}}\cdot\]
\[\exp\br{\frac{\dim\br{\chi}\chi\br{Y}}{\chi\br{X_{d}}}
\sum\limits_{m=1}^{\frac{n}{2}}c_{-m}\frac{s^{2m}}{m!}\br{\sum\limits_{r=1}^{m-1}\frac{1}{r}-2\sum\limits_{r=1}^{2m-1}\frac{1}{r}}},\]
\newline
where the coefficients $c_{k}$ are defined by the asymptotic
expansion

\[\TRR{e^{-tA_{d}\br{\gamma,\sigma}^{2}}}\overset{t\rightarrow 0}{\sim}\sum\limits_{k=-\frac{n}{2}}^{\infty}c_{k}t^{k}.\]
\newline
Hence, (see, \cite[pp.~120--122]{Bunke})

\[Z_{S,\chi}\br{-s,\sigma}=Z_{S,\chi}\br{s,\sigma}\cdot\br{\frac{\det\br{A_{d}\br{\gamma,\sigma}-s}}{\det\br{A_{d}\br{\gamma,\sigma}+s}}}^{-\frac{2\dim\br{\chi}\chi\br{Y}}{\chi\br{X_{d}}}}=\]
\[Z_{S,\chi}\br{s,\sigma}\cdot\br{\frac{D^{+}(s)}{D^{-}(s)}}^{-\frac{2\dim\br{\chi}\chi\br{Y}}{\chi\br{X_{d}}}}=\]
\[Z_{S,\chi}\br{s,\sigma}\cdot\exp\br{-\frac{\pi}{T}\int\limits_{0}^{s}P_{\sigma}\br{w}\left\{%
\begin{array}{ll}
    \tan\br{\frac{\pi w}{T}}, & \hbox{$\epsilon_{\sigma}=\frac{1}{2}$} \\
    -\cot\br{\frac{\pi w}{T}}, & \hbox{\hbox{$\epsilon_{\sigma}=0$}} \\
\end{array}%
\right\}dw}^{-\frac{2\dim\br{\chi}\chi\br{Y}}{\chi\br{X_{d}}}}.\]
\newline
This completes the proof.
\end{proof}\

\begin{lem}{4.2.}\label{IV.2}
\textit{If $\gamma$ is $\sigma\,-$ admissible, then}\\
\[P_{\sigma}\br{w}=\sum\limits_{k=0}^{\frac{n}{2}-1}p_{n-2k-1}w^{n-2k-1},\]
\textit{where}
\end{lem}
\[p_{n-2k-1}=\frac{2T}{\br{\frac{n}{2}-k-1}!}c_{-\br{\frac{n}{2}-k}},\,\,k=0,1,...,\frac{n}{2}-1,\]
\[c_{-\frac{n}{2}}=\frac{\br{\frac{n}{2}-1}!}{2T}.\]

\begin{proof}
By \cite[pp.~47-48]{Bunke}, $P_{\sigma}\br{0}=0$, $P_{\sigma}\br{-w}=-P_{\sigma}\br{w}$ and $P_{\sigma}\br{w}=w\cdot Q_{\sigma}\br{w}$, where $Q_{\sigma}$ is an even polynomial. Hence, $P_{\sigma}$ is an odd polynomial. Moreover, $P_{\sigma}\br{w}\in\mathbb{R}\SqBr{w}$ is a monic polynomial of degree $n-1$ (see, e.g., \cite[pp.~17-19]{Brocker}, \cite[pp.~240-243]{Wakayama}).

Put
\[P_{\sigma}\br{w}=\sum\limits_{k=0}^{\frac{n}{2}-1}p_{n-2k-1}w^{n-2k-1},\,\,p_{n-1}=1.\]

By \cite[p.~118]{Bunke}, $Q_{\sigma}\br{w}=\sum\limits_{k=0}^{\frac{n}{2}-1}q_{n-2k-2}w^{n-2k-2}$, where $q_{2i}=\frac{2T}{i!}c_{-\br{i+1}}$, $i=0,1,...,\frac{n}{2}-1$. In other words,

\[p_{n-2k-1}=q_{n-2k-2}=\frac{2T}{\br{\frac{n}{2}-k-1}!}c_{-\br{\frac{n}{2}-k}},\,\,k=0,1,...,\frac{n}{2}-1.\]
\newline
This completes the proof.
\end{proof}\

\begin{teorem}{4.3.}\label{IV.3}
\textit{Suppose $\sigma_{1}<0$ fixed. If $\gamma$ is $\sigma\,-$ admissible, then}\\
\[Z_{S,\chi}\br{\sigma_{1}+it,\sigma}=f(t)e^{g(t)}\cdot Z_{S,\chi}\br{-\sigma_{1}-it,\sigma},\,\abs{t}\rightarrow\infty,\]
\textit{where}
\end{teorem}
\[f(t)=\]
\[\exp\br{-\sum\limits_{k=0}^{\frac{n}{2}-1}p_{n-2k-1}\frac{\abs{t}}{t}\frac{K\II}{n-2k}\sum\limits_{l=0}^{\frac{n}{2}-k}\binom{n-2k}{2l}(-1)^{l}\sigma_{1}^{n-2k-2l}t^{2l}+O\br{1}},\]
\[g(t)=-\sum\limits_{k=0}^{\frac{n}{2}-1}p_{n-2k-1}\frac{K}{n-2k}\sum\limits_{l=1}^{\frac{n}{2}-k}\binom{n-2k}{2l-1}(-1)^{l}\sigma_{1}^{n-2k-2l+1}\abs{t}^{2l-1}.\]

\begin{proof}
Suppose $\epsilon_{\sigma}=\frac{1}{2}$. By Lemma \ref{IV.2}.2,

\[K\int\limits_{0}^{s}P_{\sigma}\br{w}\tan\br{\frac{\pi w}{T}}dw=K\sum\limits_{k=0}^{\frac{n}{2}-1}p_{n-2k-1}\int\limits_{0}^{s}w^{n-2k-1}\tan\br{\frac{\pi w}{T}}dw.\]
\newline
Here, we assume that the integration will be carried out along the line segment joining the origin to $s$ (see, \cite[p.~211]{Randol2}). As $\abs{t}\rightarrow\infty$,

\[\tan\pi\br{\sigma_{1}+\II t}=\II\frac{t}{\abs{t}}+O\br{e^{-2\pi\abs{t}}}.\]
\newline
Hence, at points on a vertical line away from the real axis, one has

\begin{equation}\label{44.1}
K\int\limits_{0}^{s}P_{\sigma}\br{w}\tan\br{\frac{\pi w}{T}}dw=
\end{equation}
\[K\sum\limits_{k=0}^{\frac{n}{2}-1}p_{n-2k-1}\br{\frac{\sigma_{1}}{t}+\II}^{n-2k}\frac{t}{\abs{t}}\II\frac{t^{n-2k}}{n-2k}+\]
\[K\sum\limits_{k=0}^{\frac{n}{2}-1}p_{n-2k-1}\br{\frac{\sigma_{1}}{t}+\II}^{n-2k}\int\limits_{0}^{t}y^{n-2k-1}O\br{e^{-2\pi\frac{\abs{y}}{T}}}dy=\]
\[\sum\limits_{k=0}^{\frac{n}{2}-1}p_{n-2k-1}\frac{t}{\abs{t}}\frac{K\II}{n-2k}\sum\limits_{l=0}^{n-2k}\binom{n-2k}{l}\sigma_{1}^{n-2k-l}t^{l}\II^{l}+S=\]
\[\sum\limits_{k=0}^{\frac{n}{2}-1}p_{n-2k-1}\frac{t}{\abs{t}}\frac{K\II}{n-2k}\sum\limits_{l=0}^{\frac{n}{2}-k}\binom{n-2k}{2l}\sigma_{1}^{n-2k-2l}t^{2l}\II^{2l}+\]
\[\sum\limits_{k=0}^{\frac{n}{2}-1}p_{n-2k-1}\frac{t}{\abs{t}}\frac{K\II}{n-2k}\sum\limits_{l=1}^{\frac{n}{2}-k}\binom{n-2k}{2l-1}\sigma_{1}^{n-2k-2l+1}t^{2l-1}\II^{2l-1}+S=\]
\[\sum\limits_{k=0}^{\frac{n}{2}-1}p_{n-2k-1}\frac{t}{\abs{t}}\frac{K\II}{n-2k}\sum\limits_{l=0}^{\frac{n}{2}-k}\binom{n-2k}{2l}(-1)^{l}\sigma_{1}^{n-2k-2l}t^{2l}+\]
\[\sum\limits_{k=0}^{\frac{n}{2}-1}p_{n-2k-1}\frac{K}{n-2k}\sum\limits_{l=1}^{\frac{n}{2}-k}\binom{n-2k}{2l-1}(-1)^{l}\sigma_{1}^{n-2k-2l+1}\abs{t}^{2l-1}+S.\]
\newline
Assume $t>0$. We have

\[\int\limits_{0}^{t}y^{n-2k-1}O\br{e^{-2\pi\frac{\abs{y}}{T}}}dy=O\br{\int\limits_{0}^{t}y^{n-2k-1}e^{-2\pi\frac{y}{T}}dy}.\]
\newline
Applying integration by parts $n-2k-1$ times, one easily obtains that

\[\int\limits_{0}^{t}y^{n-2k-1}e^{-2\pi\frac{y}{T}}dy=O\br{1}.\]
Hence,
\[\int\limits_{0}^{t}y^{n-2k-1}O\br{e^{-2\pi\frac{\abs{y}}{T}}}dy=O\br{1}.\]
If $t<0$, then

\[\int\limits_{0}^{t}y^{n-2k-1}O\br{e^{-2\pi\frac{\abs{y}}{T}}}dy=\int\limits_{0}^{-t}y^{n-2k-1}O\br{e^{-2\pi\frac{\abs{y}}{T}}}dy=O\br{1}.\]
\newline
We conclude, $S=O\br{1}$. Therefore, by (\ref{44.1}),

\begin{equation}\label{44.2}
-K\int\limits_{0}^{s}P_{\sigma}\br{w}\tan\br{\frac{\pi w}{T}}dw=
\end{equation}
\[-\sum\limits_{k=0}^{\frac{n}{2}-1}p_{n-2k-1}\frac{t}{\abs{t}}\frac{K\II}{n-2k}\sum\limits_{l=0}^{\frac{n}{2}-k}\binom{n-2k}{2l}(-1)^{l}\sigma_{1}^{n-2k-2l}t^{2l}-\]
\[\sum\limits_{k=0}^{\frac{n}{2}-1}p_{n-2k-1}\frac{K}{n-2k}\sum\limits_{l=1}^{\frac{n}{2}-k}\binom{n-2k}{2l-1}(-1)^{l}\sigma_{1}^{n-2k-2l+1}\abs{t}^{2l-1}+O\br{1}.\]
\newline
Now, suppose $\epsilon_{\sigma}=0$. It is not difficult to verify that

\[\cot\pi\br{\sigma_{1}+\II t}=-\II\frac{t}{\abs{t}}+O\br{e^{-2\pi\abs{t}}}\]
\newline
as $\abs{t}\rightarrow\infty$. We have

\begin{equation}\label{44.3}
-K\int\limits_{0}^{s}P_{\sigma}\br{w}\br{-\cot\br{\frac{\pi w}{T}}}dw=K\int\limits_{0}^{s}P_{\sigma}\br{w}\cot\br{\frac{\pi w}{T}}dw=
\end{equation}
\[K\sum\limits_{k=0}^{\frac{n}{2}-1}p_{n-2k-1}\int\limits_{0}^{s}w^{n-2k-1}\cot\br{\frac{\pi w}{T}}dw=\]
\[-K\sum\limits_{k=0}^{\frac{n}{2}-1}p_{n-2k-1}\br{\frac{\sigma_{1}}{t}+\II}^{n-2k}\frac{t}{\abs{t}}\II\frac{t^{n-2k}}{n-2k}+S=\]
\[-\sum\limits_{k=0}^{\frac{n}{2}-1}p_{n-2k-1}\frac{t}{\abs{t}}\frac{K\II}{n-2k}\sum\limits_{l=0}^{\frac{n}{2}-k}\binom{n-2k}{2l}(-1)^{l}\sigma_{1}^{n-2k-2l}t^{2l}-\]
\[\sum\limits_{k=0}^{\frac{n}{2}-1}p_{n-2k-1}\frac{K}{n-2k}\sum\limits_{l=1}^{\frac{n}{2}-k}\binom{n-2k}{2l-1}(-1)^{l}\sigma_{1}^{n-2k-2l+1}\abs{t}^{2l-1}+O\br{1}.\]
\newline
Combining (\ref{44.2}) and (\ref{44.3}), we finally obtain

\begin{equation}\label{44.4}
-K\int\limits_{0}^{s}P_{\sigma}\br{w}\left\{%
\begin{array}{ll}
    \tan\br{\frac{\pi w}{T}}, & \hbox{$\epsilon_{\sigma}=\frac{1}{2}$} \\
    -\cot\br{\frac{\pi w}{T}}, & \hbox{\hbox{$\epsilon_{\sigma}=0$}} \\
\end{array}%
\right\}dw=
\end{equation}
\[-\sum\limits_{k=0}^{\frac{n}{2}-1}p_{n-2k-1}\frac{t}{\abs{t}}\frac{K\II}{n-2k}\sum\limits_{l=0}^{\frac{n}{2}-k}\binom{n-2k}{2l}(-1)^{l}\sigma_{1}^{n-2k-2l}t^{2l}-\]
\[\sum\limits_{k=0}^{\frac{n}{2}-1}p_{n-2k-1}\frac{K}{n-2k}\sum\limits_{l=1}^{\frac{n}{2}-k}\binom{n-2k}{2l-1}(-1)^{l}\sigma_{1}^{n-2k-2l+1}\abs{t}^{2l-1}+O\br{1}.\]
\newline
The theorem now follows from (\ref{44.4}) and Lemma \ref{IV.2}.1.
\end{proof}\

\begin{lem}{4.4.}\label{IV.4}
\textit{If $\gamma$ is $\sigma\,-$ admissible, then}
\newline
\[\abs{Z_{S}^{i}\br{\sigma_{1}+\II t}}=e^{O\br{\abs{t}^{n-1}}}\]
\newline
\textit{uniformly in any bounded strip $b_{1}\leq\sigma_{1}\leq b_{2}$ for $i=1,2$.}
\end{lem}
\begin{proof}
Let $c>\max\set{\rho,\abs{b_{1}},\abs{b_{2}}}$. It is enough to prove the assertion for a wider strip $-c\leq\sigma_{1}\leq c$.

By (\ref{3.3}), $Z_{S}^{1}\br{s}$ and $Z_{S}^{2}\br{s}$ are of order at most $n$. Furthermore, the infinite product (\ref{3.0}) defining $Z_{S,\chi}\br{s,\sigma}$ converges for $\RE{\br{s}}>\rho$ (see, \cite[pp.~98-99]{Bunke}). Hence, the lemma is implied by the Phragm\'{e}n-Lind\"{o}lef theorem and Theorem \ref{IV.3}.3.
\end{proof}\

\begin{teorem}{4.5.}\label{IV.5}
\textit{If $\gamma$ is $\sigma\,-$ admissible, then}
\newline
\[N\br{t}=\frac{K}{2\pi}\sum\limits_{k=0}^{\frac{n}{2}-1}\br{-1}^{\frac{n}{2}-k}p_{n-2k-1}\frac{t^{n-2k}}{n-2k}+\frac{1}{\pi}S\br{t}+O\br{1},\]
\newline
\textit{where $N\br{t}$ denotes the number of singularities of $Z_{S,\chi}\br{s,\sigma}$ on the interval $ix$, $0<x<t$ and $S\br{t}$ is the variation of the argument of $Z_{S,\chi}\br{s,\sigma}$ along $C$. Here, $C$ denotes the portion of $\partial R$ consisting of the vertical segment from $a$ to $a+\II t$ plus horizontal segment from $a+\II t$ to $\II t$, where $R$ is the rectangle defined by the inequalities $-a\leq\RE{\br{s}}\leq a$, $-t\leq\IM{\br{s}}\leq t$ for some $a>\rho$.}
\end{teorem}
\begin{proof}
We follow \cite{Randol2}.

Define $\xi\br{s}=\br{Z_{S,\chi}\br{s,\sigma}}^{2}e^{\phi\br{s}}$, where

\[\phi\br{s}=K\int\limits_{0}^{s}P_{\sigma}\br{w}\left\{%
\begin{array}{ll}
    \tan\br{\frac{\pi w}{T}}, & \hbox{$\epsilon_{\sigma}=\frac{1}{2}$} \\
    -\cot\br{\frac{\pi w}{T}}, & \hbox{\hbox{$\epsilon_{\sigma}=0$}} \\
\end{array}%
\right\}dw.\]
\newline
Here, as earlier, we specify $\phi\br{s}$ in the open upper and lower half-planes to be the value obtained by carrying out the integration along the line segment joining the origin to $s$. Furthermore, if $\epsilon_{\sigma}=\frac{1}{2}$ resp. $\epsilon_{\sigma}=0$ and $s$ is on the real line, $s\neq\pm\frac{T}{2}, \pm\frac{3T}{2}, \pm\frac{5T}{2},...$ resp. $s\neq\ 0, \pm T, \pm 2T,...$, we define $\phi\br{s}$ by the requirement of continuity as $s$ is approached from the upper half-plane.

By Lemma \ref{IV.1}.1,
\[Z_{S,\chi}\br{-s,\sigma}=e^{\phi\br{s}}Z_{S,\chi}\br{s,\sigma}.\]
Hence,
\[\xi\br{-s}=\br{Z_{S,\chi}\br{-s,\sigma}}^{2}e^{\phi\br{-s}}=\]

\[\br{Z_{S,\chi}\br{s,\sigma}}^{2}e^{2\phi\br{s}}e^{-\phi\br{s}}=\br{Z_{S,\chi}\br{s,\sigma}}^{2}e^{\phi\br{s}}=\xi\br{s}.\]
\newline
As usual, $\xi\br{s}$ is real on the real axis and so $\overline{\xi\br{s}}=\xi\br{\overline{s}}$.

Assume that $t$ is selected so that no singularity of $Z_{S,\chi}\br{s,\sigma}$ occurs on the boundary of $R$. We have,

\[N\br{t}=\frac{1}{4}\cdot\frac{1}{2\pi\II}\int\limits_{\partial R}\frac{\xi^{'}\br{s}}{\xi\br{s}}ds-\frac{1}{2}N_{0}=\frac{1}{4}\cdot\frac{1}{2\pi}\IM{\br{\int\limits_{\partial R}\frac{\xi^{'}\br{s}}{\xi\br{s}}ds}}-\frac{1}{2}N_{0},\]
\newline
where $N_{0}=O\br{1}$ is the number of singularities of $Z_{S,\chi}\br{s,\sigma}$ on the real line.

From the functional equation for $\xi\br{s}$ and the fact that $\overline{\xi\br{s}}=\xi\br{\overline{s}}$, it follows in the classical way that

\[N\br{t}=\frac{1}{2\pi}\IM{\br{\int\limits_{C}\frac{\xi^{'}\br{s}}{\xi\br{s}}ds}}+O\br{1}.\]
By (\ref{44.4}),
\[\phi\br{\sigma_{1}+\II t}=\]

\[\sum\limits_{k=0}^{\frac{n}{2}-1}p_{n-2k-1}\frac{t}{\abs{t}}\frac{K\II}{n-2k}\sum\limits_{l=0}^{\frac{n}{2}-k}\binom{n-2k}{2l}(-1)^{l}\sigma_{1}^{n-2k-2l}t^{2l}+\]
\[\sum\limits_{k=0}^{\frac{n}{2}-1}p_{n-2k-1}\frac{K}{n-2k}\sum\limits_{l=1}^{\frac{n}{2}-k}\binom{n-2k}{2l-1}(-1)^{l}\sigma_{1}^{n-2k-2l+1}\abs{t}^{2l-1}+O\br{1}.\]
\newline
Now,
\[\frac{\xi^{'}\br{s}}{\xi\br{s}}=\phi^{'}\br{s}+2\frac{Z^{'}_{S,\chi}\br{s,\sigma}}{Z_{S,\chi}\br{s,\sigma}},\]
so
\[N\br{t}=\frac{1}{2\pi}\IM{\br{\int\limits_{C}\phi^{'}\br{s}ds}}+\frac{1}{\pi}\IM{\br{\int\limits_{C}\frac{Z^{'}_{S,\chi}\br{s,\sigma}}{Z_{S,\chi}\br{s,\sigma}}ds}}+O\br{1}=\]

\[\frac{1}{2\pi}\IM{\br{\phi\br{\II t}-\phi\br{a}}}+\frac{1}{\pi}S\br{t}+O\br{1}=\]

\[\frac{1}{2\pi}\IM{\phi\br{\II t}}+\frac{1}{\pi}S\br{t}-\frac{1}{2\pi}\phi\br{a}+O\br{1}=\]

\[\frac{K}{2\pi}\sum\limits_{k=0}^{\frac{n}{2}-1}\br{-1}^{\frac{n}{2}-k}p_{n-2k-1}\frac{t^{n-2k}}{n-2k}+\frac{1}{\pi}S\br{t}+O\br{1}.\]
\newline
This completes the proof.
\end{proof}\

\begin{lem}{4.6.}\label{IV.6}
\textit{If $\gamma$ is $\sigma\,-$ admissible, then}
\newline
\[S\br{t}=O\br{t^{n-1}}.\]
\end{lem}

\begin{proof}
First, we extend the definition of $S\br{t}$ to those values of $t$ for which $\II t$ is a pole or zero of $Z_{S,\chi}\br{s,\sigma}$ by defining it to be

\[\lim_{\varepsilon\rightarrow 0}\frac{1}{2}\br{S\br{t+\varepsilon}+S\br{t-\varepsilon}}.\]
Then, we have
\[S\br{t}=h\br{t}+O\br{1},\]
\newline
where $h\br{t}$ is the variation of the argument of $Z_{S,\chi}\br{s,\sigma}$ along the segment from $a+\II t$ to $\II t$.

Proceeding in accordance with a custom, one easily concludes that

\begin{equation}\label{44.5}
h\br{t}=O\br{\int\limits_{\partial S}\log\abs{Z_{S,\chi}\br{w+\II t,\sigma}}+\log\abs{Z_{S,\chi}\br{w-\II t,\sigma}}dw}=
\end{equation}
\[O\br{\sum\limits_{i=1,2}\int\limits_{\partial S}\log\abs{Z^{i}_{S}\br{w+\II t}}dw+\int\limits_{\partial S}\log\abs{Z^{i}_{S}\br{w-\II t}}dw},\]
\newline
where $S$ is the closed disc, centered at $a$, of radius $a+\frac{1}{4}$.

Now, the assertion follows from Lemma \ref{IV.4}.4 and (\ref{44.5}).
\end{proof}\

\begin{cor}{4.7.}\label{IV.7}
\textit{If $\gamma$ is $\sigma\,-$ admissible, then}

\[N\br{t}=\frac{\dim\br{\chi}\vol{\br{Y}}}{nT\vol{\br{X_{d}}}}t^{n}+O\br{t^{n-1}}.\]
\end{cor}\

\begin{proof}
An immediate consequence of (\ref{1111}), Theorem \ref{IV.5}.5 and Lemma \ref{IV.6}.6.
\end{proof}\

\begin{cor}{4.8.}\label{IV.8}
\textit{Let $-\rho\leq a\leq b\leq\rho$. If $\gamma$ is $\sigma\,-$ admissible, then there exists a constant $C$ such that}

\[N_{R}\br{t}=Ct^{n}+O\br{t^{n-1}},\]
\newline
\textit{where $N_{R}\br{t}$ denotes the number of singularities of $Z_{R,\chi}\br{s,\sigma}$ in the rectangle $a\leq\RE{\br{s}}\leq b$, $0<\IM{\br{s}}<t$.}
\end{cor}

\begin{proof}
Trivial consequence of the formula (\ref{3.1}) and Corollary \ref{IV.7}.7.
\end{proof}\

\begin{remm}{4.9.}\label{IV.9}
Precise estimates of the number of singularities of the zeta functions represent an important tool which plays the key role in achieving more refined error terms in the prime geodesic theorem (see, \cite{Randol1}, \cite{AG0}, \cite{Park}). Such counting functions may also be used elsewhere.
\end{remm}\
\newline
\begin{rem}{}
We thank Professor Martin Olbrich for his valuable comments. 
\end{rem}

\end{document}